\documentclass[paper=a4,titlepage,parskip=half,%
DIV10,onecolumn,notitlepage]{scrartcl}

\usepackage[english]{babel}																	
\usepackage[utf8]{inputenc}																
\usepackage[T1]{fontenc}																		
\usepackage{array}																					
\usepackage{graphics}																
\usepackage{color}																					
\usepackage[english=american]{csquotes}											
\usepackage{typearea}																				
\usepackage{amsmath, amstext}																
\usepackage[autooneside]{scrpage2}													
\usepackage{enumerate}																			
\usepackage{authblk}																				
\usepackage{listings}																				
\usepackage[																								
margin=5pt,labelfont=bf,parskip=5pt]{caption}								
\usepackage{colortbl}
\usepackage{booktabs}

\usepackage[bitstream-charter]{mathdesign}									
\usepackage[scaled]{berasans}																
\usepackage[scaled]{beramono}																

\automark{section}																					
\setkomafont{pageheadfoot}{\scriptsize}
\addtokomafont{sectioning}{\rmfamily}

\usepackage{tikz}
\usepackage{nndoubleletters}																

\usepackage[hyperref,thmmarks,amsmath,thref]{ntheorem}
 
\theoremstyle{break}
\theoremnumbering{arabic}
\theorembodyfont{\normalfont} 
\newtheorem{thm}{Theorem}
\newtheorem{lem}[thm]{Lemma}   
\newtheorem{cor}[thm]{Corollary}   
\newtheorem{prop}[thm]{Proposition}  
\newtheorem{conj}[thm]{Conjecture}
\newtheorem{remark}[thm]{Remark}
\newtheorem{defi}[thm]{Definition}
\theorembodyfont{\normalfont}  

\theoremstyle{nonumberplain} 
\theoremsymbol{\ensuremath{\square}}
\newtheorem{proof}{Proof.}

\theoremsymbol{}

\usepackage{hyperref}

\allowdisplaybreaks[4]

\DeclareMathOperator{\aff}{aff}
\DeclareMathOperator{\conv}{conv}
\DeclareMathOperator{\NH}{NH}

\newcommand{\Ol}{\mathcal O}

\newcommand{\mach}[1]{|#1|}
\newcommand{\witho}{\backslash}
\newcommand{\shadow}{\mathcal S}
\newcommand{\scalefactor}{.6}

\newcommand{\embH}[2]{\mathcal{H}_{#1,#2}}
\newcommand{\embHP}[2]{\mathcal{H}_{#1,#2}^{\text{\normalfont\textsf{PL}}}}

\newcommand{\str}{^{\text{\normalfont\textsf{s}}}}
\newcommand{\wea}{^{\text{\normalfont\textsf{w}}}}

\renewcommand{\phi}{\varphi}
\renewcommand{\epsilon}{\varepsilon}

\definecolor{rred}{rgb}{.8,0,0}

\hypersetup{
pdfborder={0 0 0},
pdfpagemode=UseOutlines,
bookmarksnumbered=true,
pdftitle=Coloring d-Embeddable k-Uniform Hypergraphs,
pdfauthor={Carl Georg Heise, Konstantinos Panagiotou, Oleg Pikhurko, Anusch Taraz},
pdfsubject=Coloring d-Embeddable k-Uniform Hypergraphs,
pdfcreator=,
pdfproducer=,
}

\title{Coloring \(d\)-Embeddable \(k\)-Uniform Hypergraphs}
\newcounter{foo}
\author[1,\fnsymbol{foo}]{Carl Georg Heise}\addtocounter{foo}{1}
\author[2]{Konstantinos Panagiotou}
\author[3,\fnsymbol{foo}]{Oleg Pikhurko}\addtocounter{foo}{1}
\author[1,\fnsymbol{foo}]{Anusch Taraz}\addtocounter{foo}{1}
\affil[1]{\footnotesize Institut für Mathematik, Technische Universität Hamburg-Harburg, Germany, \{carl.georg.heise, taraz\}@tuhh.de}
\affil[2]{\footnotesize Mathematisches Institut, Ludwig-Maximilians-Universität München, Germany, kpanagio@math.lmu.de}
\affil[3]{\footnotesize Mathematics Institute and DIMAP, University of Warwick, Coventry
, UK, o.pikhurko@warwick.ac.uk}

\date{November 27, 2014}

\clearscrheadfoot																						
\ohead{\pagemark}									
\ihead{Coloring \(d\)-Embeddable \(k\)-Uniform Hypergraphs}
\newpagestyle{first}{{DOI 10.1007/s00454-014-9641-2\hfill Discrete Comput Geom (2014) 52:663--679}{DOI 10.1007/s00454-014-9641-2\hfill Discrete Comput Geom (2014) 52:663--679}{DOI 10.1007/s00454-014-9641-2\hfill Discrete Comput Geom (2014) 52:663--679}}{{\hfill}{\hfill}{\hfill}}

\begin{document}
\maketitle
\thispagestyle{first}

\let\oldthefootnote\thefootnote
\renewcommand{\thefootnote}{\fnsymbol{footnote}}
\footnotetext[2]{Partially supported by the ENB graduate program TopMath and DFG grant GR 993/10-1. The author gratefully acknowledges the support of the TUM Graduate School's Thematic
Graduate Center TopMath at the Technische Universität München}
\footnotetext[3]{Partially supported by the Engineering and Physical Sciences Research Council (grant EP/K012045/1), the Alexander von Humboldt Foundation, and the European Research Council (grant No. 306493)}
\footnotetext[4]{Partially supported by DFG grant TA 309/2-2.}
\let\thefootnote\oldthefootnote
\section*{Abstract}

This paper extends the scenario of the Four Color Theorem in the following way. Let $\embH{d}{k}$ be the set of all $k$-uniform hypergraphs that can be (linearly) embedded into $\RR^d$. We investigate lower and upper bounds on the maximum (weak) chromatic number of hypergraphs in $\embH{d}{k}$. For example, we can prove that for $d\geq3$ there are hypergraphs in $\embH{2d-3}{d}$ on $n$ vertices whose chromatic number is $\Omega(\log n/\log\log n)$, whereas the chromatic number for $n$-vertex hypergraphs in  $\embH{d}{d}$ is bounded by $\Ol(n^{(d-2)/(d-1)})$ for $d\geq3$.
\section{Introduction}
\label{sec:intro}

The Four Color Theorem \cite{FCT1,FCT2} asserts that every graph that is embeddable in the plane has chromatic number at most four. This question has been one of the driving forces in Discrete Mathematics and its theme has inspired many variations.
For example, the chromatic number of graphs that are embedabble into a surface of fixed genus has been intensively studied by Heawood~\cite{Heawood}, Ringel and Youngs~\cite{ringelyoungs}, and many others.  

In this paper, we consider $k$-uniform hypergraphs that are embeddable 
into $\RR^d$ in such a way that their edges do not intersect (see Definition~\ref{def:dembed} below). For $k=d=2$ the problem specializes to graph planarity. For $k=2$ and $d\geq3$ it is not a very interesting question because for any $n\in\NN$ the vertices of the complete graph $K_n$ can be embedded into $\RR^3$ using the embedding
\begin{equation}
\label{eq:knR3}
\phi(v_i)=\left(i,i^2,i^3\right)\quad\forall i\in\{1,\ldots,n\}.
\end{equation}
It is a well known property of the moment curve $t\mapsto(t,t^2,t^3)$ that any two edges between four distinct vertices do not intersect (see Proposition~\ref{pro:shephardPlane}).

As a consequence, we now focus our attention on hypergraphs, which are in general not embeddable into any specific dimension.
Some properties of these hypergraphs (or more generally simplicial complexes) have been investigated (see e.\,g.~\cite{Flores1,Flores,WagnerHard,Menger,schild,VanKampen}), 
but to our surprise, we have not been able to find any previously established results which bound their chromatic number. However, Grünbaum and Sarkaria (see \cite{grunb,sarkaria})
have considered a different generalization of graph colorings to simplicial complexes by coloring faces. 
They also bound this face-chromatic number subject
to embeddability constraints.

Before we can state our main results, we quickly recall and introduce some useful notation. 
We say that $H=(V,E)$ is a \emph{$k$-uniform hypergraph} if the vertex set $V$ is a finite set and the edge set $E$ consists of $k$-element subsets of $V$, i.\,e. $E\subseteq \binom{V}{k}$. For any hypergraph $H$, we denote by $V(H)$ the vertex set of $H$ and by $E(H)$ its edge set.
We define 
\[K_n^{(k)}:=\left(\{1,2,\ldots,n\},\binom{\{1,2,\ldots,n\}}{k}\right)\]
and call any hypergraph isomorphic to $K_n^{(k)}$ a \emph{complete $k$-uniform hypergraph of order $n$}.

Let $H$ be a $k$-uniform hypergraph. A function
$\kappa:V(H)\to\{1,\ldots,c\}$
is said to be a \emph{weak $c$-coloring} if for all $e\in E(H)$ the property $\mach{\kappa(e)}>1$ holds.
The function $\kappa$ is said to be a \emph{strong $c$-coloring} if $\mach{\kappa(e)}=k$ for all $e\in E(H)$.
The \emph{weak/strong chromatic number} of $H$ is defined as the minimum $c\in\NN$ such that there exists a weak/strong coloring of $H$ with $c$ colors. The chromatic number of $H$ is denoted by $\chi\wea(H)$ and $\chi\str(H)$, respectively.
Obviously, for graphs, weak and strong colorings are equivalent.

We next define what we mean when we say that a hypergraph is embeddable into $\RR^d$. Here, $\aff$ denotes the affine hull of a set of points and $\conv$ the convex hull. 
\begin{defi}[$d$-embeddings]
\label{def:dembed}
Let $H$ be a $k$-uniform hypergraph and $d\in\NN$. A (linear) \emph{embedding} of $H$ into $\RR^d$ is a function $\phi:V(H)\to\RR^d$, where $\phi(A)$ for $A\subseteq V(H)$ is to be interpreted pointwise, such that
\begin{itemize}
	\item $\dim\aff \phi(e)=k-1$ for all $e\in E(H)${} and
	\item $\conv\phi(e_1)\cap \conv\phi(e_2)=\conv\phi\left(e_1\cap e_2\right)${} for all $e_1, e_2\in E(H)$.
\end{itemize}
\end{defi}

The first property is needed to exclude functions mapping the vertices of one edge to affinely non-independent points. The second guarantees that the embedded edges only intersect in the convex hull of their common vertices. Note that the inclusion from left to right always holds.
A $k$-uniform hypergraph $H$ is said to be \emph{$d$-embeddable} if there exists an embedding of $H$ into $\RR^d$.
Also, we denote by $\embH{d}{k}$ the set of all $d$-embeddable $k$-uniform hypergraphs.

One can easily see that our definition of 2-embeddability coincides with the classical concept of planarity \cite{Fary}. Note that in general there are several other notions of embeddability. The most popular thereof are piecewise linear embeddings and general topological embeddings. A short and comprehensive introduction is given in Section 1 in \cite{WagnerHard}. Furthermore, there exist some quite different concepts of generalizing embeddability for hypergraphs in the literature, for example \emph{hypergraph imbeddings} \cite[Chapter 13]{whiteImbed}.

We have decided to focus on linear embeddings, as they lead to a very accessible type of geometry and, at least in theory, the decision problem of whether a given $k$-uniform hypergraph is $d$-embeddable is decidable and in PSPACE \cite{Renegar}. One can show that the aforementioned three types of embeddings are equivalent only in the less than 3-dimensional case (see e.\,g.~\cite{Bokowski,Brehm}), although piecewise linear and topological embeddability coincides if $d-k\geq 2$ or $(d,k)=(3,3)$, see \cite{bryant}. Since piecewise linear and topological embeddings are more general than linear embeddings, all lower bounds for chromatic numbers can easily be transferred. Furthermore, we prove all our results on upper bounds for piecewise linear embeddings (and thus also for topological embeddings if $d-k\geq 2$ or $(d,k)=(3,3)$) except for one case (namely Theorem~\ref{thm:wup3}).

We can now give a summary of our main results in the following Tables~\ref{tab:currWeakL} and \ref{tab:currWeakU}, which contain upper or lower bounds for the maximum weak chromatic number of a $d$-embeddable $k$-uniform hypergraph on $n$ vertices. 
All results which only follow non-trivially from prior knowledge are indexed with a theorem number from which they can be derived.

\definecolor{agray}{gray}{.85}


\begin{table}[htbp]
	\centering
		\begin{tabular}{lcccccc}
		\arrayrulecolor{black}
		\toprule
		$d\diagdown k$  & 2 & 3 & 4 & 5 & 6 & 7\\
		\midrule
		\textbf{1} & 2 & 1 & 1 & 1 & 1 & 1 \\
		\arrayrulecolor{agray}
		\hline
		\textbf{2} & 4 & 2 & 1 & 1 & 1 & 1 \\
		\hline
		\textbf{3} & $n$ & $\Omega\left(\frac{\log n}{\log\log n}\right)_{\langle\ref{thm:wlo1}\rangle}$ & 1 & 1 & 1 & 1 \\
		\hline
		\textbf{4} & $n$ & $\Omega\left(\frac{\log n}{\log\log n}\right)_{\langle\ref{thm:wlo1}\rangle}$ & 1 & 1 & 1 & 1 \\
		\hline
		\textbf{5} & $n$ & $\lceil n/2\rceil$ & $\Omega\left(\frac{\log n}{\log\log n}\right)_{\langle\ref{thm:wlo2}\rangle}$ & 1 & 1 & 1  \\
		\hline
		\textbf{6} & $n$ & $\lceil n/2\rceil$ & $\Omega\left(\frac{\log n}{\log\log n}\right)_{\langle\ref{thm:wlo2}\rangle}$ & 1 & 1 & 1\\
		\hline
		\textbf{7} & $n$ & $\lceil n/2\rceil$ & $\lceil n/3\rceil$ & $\Omega\left(\frac{\log n}{\log\log n}\right)_{\langle\ref{thm:wlo2}\rangle}$ & 1 & 1\\
		\hline
		\textbf{8} & $n$ & $\lceil n/2\rceil$ & $\lceil n/3\rceil$ & $\Omega\left(\frac{\log n}{\log\log n}\right)_{\langle\ref{thm:wlo2}\rangle}$ & 1 & 1 \\
		\arrayrulecolor{black}
		\bottomrule		
		\end{tabular}
	\caption{Currently known lower bounds for the maximum weak chromatic number of a $d$-embeddable $k$-uniform hypergraph on $n$ vertices as $n\to\infty$. The number in chevrons indicates the theorem number where we prove this bound.}
	\label{tab:currWeakL}
\end{table}

\begin{table}[htbp]
	\centering
		\begin{tabular}{lcccccc}
		\arrayrulecolor{black}
		\toprule
		$d\diagdown k$  & 2 & 3 & 4 & 5 & 6 & 7\\
		\midrule
		\textbf{1} & 2 & 1 & 1 & 1 & 1 & 1 \\
		\arrayrulecolor{agray}
		\hline
		\textbf{2} & 4 & 2 & 1 & 1 & 1 & 1 \\
		\hline
		\textbf{3} & $n$ & $\Ol(n^{1/2})_{\langle\ref{thm:wup1}\rangle}$ & $\Ol(n^{1/2})_{\langle\ref{thm:wup1}\rangle}$ & 1 & 1 & 1 \\
		\hline
		\textbf{4} & $n$ & $\lceil n/2\rceil$ & $\Ol(n^{2/3})_{\langle\ref{thm:wup1}\rangle}$ & $\Ol(n^{1/2})_{\langle\ref{thm:wup3}\rangle}$ & 1 & 1 \\
		\hline
		\textbf{5} & $n$ & $\lceil n/2\rceil$ & $\Ol(n^{26/27})_{\langle\ref{thm:wup2}\rangle}$ & $\Ol(n^{3/4})_{\langle\ref{thm:wup1}\rangle}$ & $\Ol(n^{3/5})_{\langle\ref{thm:wup3}\rangle}$ & 1  \\
		\hline
		\textbf{6} & $n$ & $\lceil n/2\rceil$ & $\lceil n/3\rceil$ & $\Ol(n^{35/36})_{\langle\ref{thm:wup2}\rangle}$ & $\Ol(n^{4/5})_{\langle\ref{thm:wup1}\rangle}$ & $\Ol(n^{1/2})_{\langle\ref{thm:wup3}\rangle}$\\
		\hline
		\textbf{7} & $n$ & $\lceil n/2\rceil$ & $\lceil n/3\rceil$ & $\Ol(n^{107/108})_{\langle\ref{thm:wup2}\rangle}$ & $\Ol(n^{44/45})_{\langle\ref{thm:wup2}\rangle}$ & $\Ol(n^{5/6})_{\langle\ref{thm:wup1}\rangle}$\\
		\hline
		\textbf{8} & $n$ & $\lceil n/2\rceil$ & $\lceil n/3\rceil$ & $\lceil n/4\rceil$ & $\Ol(n^{134/135})_{\langle\ref{thm:wup2}\rangle}$ & $\Ol(n^{53/54})_{\langle\ref{thm:wup2}\rangle}$ \\
		\arrayrulecolor{black}
		\bottomrule		
		\end{tabular}
	\caption{Currently known upper bounds for the maximum weak chromatic number of a $d$-embeddable $k$-uniform hypergraph on $n$ vertices as $n\to\infty$. The number in chevrons indicates the theorem number where we prove this bound.}
	\label{tab:currWeakU}
\end{table}

Considering the strong chromatic number, the question whether embeddability restricts the number of colors needed can be answered negatively by the following observation.

Let $n,d\in\NN$ such that $d\geq 3$ and $n\geq d+1$ and let $V=\{1,\ldots,n\}$. Let $\phi:\RR\to\RR^d$, $\phi(x)=(x,\ldots,x^{d+1})$ be the $(d+1)$-dimensional moment curve. Then $\phi(V)$ are the vertices of a cyclic polytope $P=\conv\phi(V)$ (see \cite{cara1,cara2,motzkin}). As $d\geq 3$, we have that $P$ is 2-neighborly \cite{gale}. Define $H(P)=(V,E(P))$ to be the $(d+1)$-uniform hypergraph with $E(P)=\{e\subseteq V:e\text{ is the set of vertices of a facet of }P\}$. Then $H(P)$ can be linearly embedded into $\RR^d$: for example, one can take the Schegel-Diagram \cite{schlegel} of $P$ with respect to some facet.

Now, choose $k\in\NN$ such that $2\leq k\leq d+1$. Following \cite[§7.1]{handbook}, for any hypergraph $H=(W,E)$, we call 
\[\shadow_k(H)=\left(W,\left\{\{w_1,\ldots,w_k\}:\{w_1,\ldots,w_k\}\subseteq e\text{ for some }e\in E\right\}\right)\]
the \emph{$k$-shadow} of $H$. As $P$ is 2-neighborly we have that $\shadow_2(H(P))=K_n$ and thus $\chi\str(H(P))=n$.
Obviously, $\shadow_2(\shadow_k(H(P)))=K_n$ and $\chi\str(\shadow_k(H(P)))=n$, too. Thus, we have demonstrated that for any $2\leq k\leq d+1\leq n$ there exists a $k$-uniform hypergraph on $n$ vertices that is linearly $d$-embeddable and has strong chromatic number $n$.

Thus, from now on, we restrict ourselves to the weak case and will always mean this when talking about a chromatic number.
To conclude the introduction, here is a rough outline for the rest of the paper.
In Section \ref{sec:embed} the general concept of embedding hypergraphs into $d$-dimensional space is discussed. We also show the embeddability of certain structures needed later on, hereby extensively using known properties of the moment curve $t\mapsto(t,t^2,t^3,\ldots,t^d)$. Then, Section~\ref{sec:weak} presents our current level of knowledge for the more difficult problem of weakly coloring hypergraphs.
\section{Embeddability}
\label{sec:embed}

The first part of this section gives insight into the structure of neighborhoods of single vertices in a hypergraph $H\in\embH{d}{k}$. We will later use this information to prove upper bounds on the number of edges in our hypergraphs. This will then yield upper bounds on the weak chromatic number.
However, we must first take a small technical detour into piecewise linear embeddings. As our hypergraphs are finite and of fixed uniformity we give a slightly simplified definition (for a more comprehensive introduction, see e.\,g. \cite{rourkeSand}).

\begin{defi}[Piecewise linear $d$-embeddings]
\label{def:dPembed}
Let $H$ be a $k$-uniform hypergraph and $D,d\in\NN$.  Let $\phi:V(H)\to\RR^{D}$ be a linear embedding of $H$ and define $\phi(H)=\bigcup_{e\in E(H)} \conv\phi(e)$.

We say $H$ is \emph{piecewise linearly embeddable} if there exists $\psi:\phi(H)\to\RR^d$ such that $\psi$ is a homeomorphism from $\phi(H)$ onto its image and there exists a (locally finite) subdivision $K$ of $\phi(H)$ (seen as a geometric simplicial complex) such that $\psi$ is affine on all elements of $K$. We call $\psi$ a \emph{piecewise linear embedding} of $H$ into $\RR^d$ and we denote by $\embHP{d}{k}$ the set of all piecewise linearly $d$-embeddable $k$-uniform hypergraphs.
\end{defi}

Note that such a $\phi$ always exists, as $H\in\embH{2k-1}{k}$ by the Menger-Nöbeling Theorem (see \cite[p. 295]{Menger} and \cite{Nobeli}). Also, Definition~\ref{def:dPembed} is independent of the choice of $\phi$.

\begin{defi}[Neighborhoods]
For a $k$-uniform hypergraph $H$ and a vertex $v\in V(H)$ we say the \emph{neighborhood} of $v$ is $N_H(v)=\{w\in V(H):w\neq v\text{ and there is an edge in $E(H)$ incident with $w$ and $v$}\}$. We define the \emph{neighborhood hypergraph} (or link) of $v\in V(H)$ to be the induced $(k-1)$-uniform hypergraph
\[\NH_H(v)=\left(N_H(v),\{e\witho \{v\}:e\in E(H),v\in e\}\right).\]
The \emph{degree} $\deg_H(v)=\deg(v)$ is the number of edges in $E(H)$ incident with $v$.
\end{defi}

\begin{lem}
\label{lem:neighbH}
For a hypergraph $H\in\embHP{d}{k}$ on $n$ vertices, $d\geq k\geq2$, and for any vertex $v$ we have that $\NH_H(v)\in\embHP{d-1}{k-1}$.
\end{lem}

\begin{proof}
Let $d\geq k\geq2$, $H\in\embHP{d}{k}$, $v\in V(H)$, and $V_v=N_H(v)$ nonempty. Then there exist $\phi:V(H)\to\RR^{2k-1}$ a linear embedding and $\psi:\phi(H)\to\RR^{d}$ a piecewise linear embedding of $H$ for some subdivision $K$ of $\phi(H)$ on whose elements $\psi$ is affine. Without restriction assume that $\phi(v)=\mathbf 0_{2k-1}$ and $\psi(\mathbf 0_{2k-1})=\mathbf 0_{d}$.

Let $H_v=(V_v\cup\{v\},\{e\in E(H):v\in e\})$ be the sub-hypergraph of $H$ of all edges containing $v$. Obviously, $\psi|\phi(H_v)$ (the restriction of $\psi$ onto $\phi(H_v)$) is a piecewise linear embedding of $H_v$ for some subdivision $K_v\subseteq K$.
Let $K^1_v=\{e\in K_v:\mathbf 0_{2k-1}\in e\}$. Then there exists an $\epsilon>0$ such that \[\epsilon\cdot\phi(H_v)\subseteq \bigcup_{e\in K^1_v} e,\] i.\,e. all points in $\epsilon\cdot\phi(H_v)$ are so close to $\mathbf 0_{2k-1}$ that they lie completely in elements of $K_v$ that contain the origin.

Then $\phi':V_v\cup\{v\}\to\RR^{2k-1}, w\mapsto \epsilon\cdot\phi(w)$ is a linear and thus $\psi|\phi'(H_v)$ a piecewise linear embedding of $H_v$ for the subdivision $K^2_v=\{e\cap\phi'(H_v):e\in K^1_v\}$.
Let $V_{K^2_v}\supseteq\phi'(V_v)$ be the set of all subdivision points of $K^2_v$ \emph{without} $\mathbf 0_{2k-1}$ and let
\[\delta=\min\left\{\left\|\psi(x)\right\|:x\in\conv(e\cap V_{K^2_v})\text{ for some }e\in K^2_v\right\}.\]


Obviously, we have that $\delta>0$. We take a regular $d$-simplex $T\subseteq\RR^d$ centered at the origin with sides of length $\delta$ and set $C=\partial T$. Due to our choice of $\delta$, all $\psi(w)$ for $w\in V_{K^2_v}$ lie outside of $T$. 
Further, for all $e\in K^2_v$ the intersection $\psi(e)\cap C$ is the union of finitely many at most $(k-2)$-dimensional simplices and homeomorphic to a $(k-2)$-dimensional simplex.
Also, as $d\geq k$, there exists a point $x\in C$ such that $x\notin\psi(e)$ for all $e\in K^2_v$.

Thus, there exists a subdivision $K^3_v$ of $K^2_v$ such that for all $e\in K^3_v$ with dimension $k-1$ we have that $\psi(e)\cap C$ is a $(k-2)$-dimensional simplex and still $\mathbf 0_{2k-1}\in e$ . We denote the set of subdivision points \emph{without} $\mathbf 0_{2k-1}$ by $V_{K^3_v}\supseteq V_{K^2_v}$.
Now, one can find a retraction $\rho:\psi(\phi'(H_v))\to\psi(\phi'(H_v))$ that maps each $\psi(w)$, $w\in V_{K^3_v}$, to the intersection point of the line segment $[\mathbf 0_d,\psi(w)]$ with $C$, such that $\rho$ is linear on all $\psi(e)$ for $e\in K^3_v$.

Set $\hat{K}=\{\conv(e\cap V_{K^3_v}):e\in K^3_v\}$ which is now a subdivision of $\phi'(\NH_H(v))\subseteq\phi'(H_v)$. Then the image of $\rho\circ(\psi|\hat{K})$ lies completely in $C\backslash\{x\}$. 

Finally, note that $C\backslash\{x\}$ is piecewise linearly homeomorphic to $\RR^{d-1}$ \cite[3.20]{rourkeSand}. Let $\gamma$ be such a (piecewise linear) homeomorphism . Then,
\[\hat{\psi}=\gamma\circ\rho\circ(\psi|\phi'(\NH_H(v)))\]
is a piecewise linear embedding of $\NH_H(v)$ into $\RR^{d-1}$ for some subdivision of $\hat{K}$ and $\NH_H(v)\in\embHP{d-1}{k-1}$.
\end{proof}

Note that it is quite plausible that a version of Lemma~\ref{lem:neighbH} for linear or general embeddings does not hold. Part (a) of the following result has previously been established by Dey and Pach for linear embeddings \cite[Theorem~3.1]{deyPach}.

\begin{lem}
\label{lem:sizeH}
a) For a hypergraph $H\in\embHP{k}{k}$ on $n$ vertices, $k\geq2$, we have that $\mach{E(H)}\leq \frac{6n^{k-1}-12n^{k-2}}{k!}$.

b) For a hypergraph $H\in\embHP{k+1}{k+1}$ on $n$ vertices, $k\geq2$, and for any vertex $v$ we have that $\deg_H(v)\leq \frac{6n^{k-1}-12n^{k-2}}{k!}$.
\end{lem}

\begin{proof}

If $k=2$, then (a) is equivalent to the fact that for $G$ planar $|E(G)|\leq 3n-6$. Given that (a) is true for some $k\geq2$, we show that~(b) holds for $k$ as well. Let $H\in\embHP{k+1}{k+1}$, $v$ one of the $n$ vertices. By Lemma~\ref{lem:neighbH}, $\NH_H(v)\in\embHP{k}{k}$. By~(a), $\mach{E(\NH_H(v))}\leq \frac{6n^{k-1}-12n^{k-2}}{k!}$ which implies $\deg_H(v)\leq \frac{6n^{k-1}-12n^{k-2}}{k!}$.

Given that (b) is true for some $k\geq2$, we show that (a) holds for $k+1$. Let $H\in\embHP{k+1}{k+1}$. Since (b) is true for every vertex $v_i$, we have \[\mach{E(H)}=\frac{\sum_{i=1}^{n}\deg_H(v_i)}{k+1}\leq\frac{n(6n^{k-1}-12n^{k-2})}{(k+1)k!}=\frac{6n^{k}-12n^{k-1}}{(k+1)!}.\]
\end{proof}

\begin{cor}
\label{cor:sizeH}
For a hypergraph $H\in\embHP{k}{k}$ on $n$ vertices, $k\geq3$, and for any edge $e\in E(H)$ there exist at most $k\frac{6n^{k-2}-12n^{k-3}}{(k-1)!}-k$ other edges adjacent to it.
\end{cor}

\begin{proof}
This follows from Lemma~\ref{lem:sizeH}, since every edge has exactly $k$ vertices and each of them has degree at most $\frac{6n^{k-2}-12n^{k-3}}{(k-1)!}$. As $e$ itself counts for the degree as well, one can subtract $k$.
\end{proof}

We need to bound the number of edges in a $d$-embeddable hypergraph to prove upper bounds for the chromatic number. The following results will also help to do this. Note that there exist much stronger conjectured bounds (see \cite[Conjecture~1.4.4]{gundert} and \cite[Conjecture~27]{kalai}). 

\begin{prop}[Gundert~{\cite[Proposition~3.3.5]{gundert}}]
\label{prop:gundert}
Let $k\geq2$. For a $k$-uniform hypergraph on $n$ vertices that is \emph{topologically} embedabble into $\RR^{2k-2}$, we have that $\mach{E(H)}< n^{k-3^{1-k}}$.
\end{prop}

\begin{cor}
\label{cor:gundert}
For a hypergraph $H\in\embHP{2k-\ell}{k}$ on $n$ vertices, $k\geq \ell\geq2$, we have that $\mach{E(H)}< \frac{(k-\ell+2)!}{k!}\cdot n^{k-3^{\ell-1-k}}$.
\end{cor}

\begin{proof}
This follows from inductively applying Lemma~\ref{lem:neighbH} and Proposition~\ref{prop:gundert}.
\end{proof}

\begin{cor}
\label{cor:sizeGundert}
For a hypergraph $H\in\embHP{2k-\ell}{k}$ on $n$ vertices, $k\geq \ell\geq3$, and for any edge $e\in E(H)$ there exist at most $\frac{(k-\ell+2)!}{(k-1)!}\cdot n^{k-1-3^{\ell-1-k}}-k$ other edges adjacent to it.
\end{cor}

\begin{proof}
This fact follows analogously to Corollary~\ref{cor:sizeH} from Corollary~\ref{cor:gundert}.
\end{proof}

\begin{thm}[Dey and Pach~\protect{\cite[Theorem 2.1]{deyPach}}]
\label{thm:deypach}
Let $k\geq2$. For a $k$-uniform hypergraph on $n$ vertices that is \emph{linearly} embedabble into $\RR^{k-1}$, we have that $\mach{E(H)}< kn^{\lceil (k-1)/2\rceil}$.
\end{thm}

\begin{cor}
\label{cor:deypach}
For a hypergraph $H\in\embH{k-1}{k}$ on $n$ vertices, $k\geq2$, and for any edge $e\in E(H)$ there exist at most $kn^{\lceil (k-1)/2\rceil}-1$ other edges adjacent to it.
\end{cor}

\begin{proof}
This fact follows obviously from Theorem~\ref{thm:deypach}.
\end{proof}

In order to find lower bounds for the chromatic number of hypergraphs later on, we need to be able to prove embeddability. The following theorem from Shephard will turn out to be very useful when embedding vertices of a hypergraph on the moment curve.

\begin{thm}[Shephard~\cite{shephard}]
\label{thm:shephard}
Let $W=\{w_1,\ldots,w_m\}\subseteq\RR^d$ be distinct points on the moment curve in that order and $P=\conv W$. We call a $q$-element subset $\{w_{i_1},w_{i_2},\ldots,w_{i_q}\}\subseteq W$ with $i_1<i_2<\cdots<i_q$ \emph{contiguous} if $i_q-i_1 = q-1$. Then $U\subseteq W$ is the set of vertices of a $(k-1)$-face of $P$ if and only if $\mach U=k$ and for some $t\geq 0$
\[U=Y_S\cup X_1\cup\cdots\cup X_t\cup Y_E,\]
where all $X_i$, $Y_S$, and $Y_E$ are \emph{contiguous} sets, $Y_S=\emptyset$ or $w_1\in Y_S$, $Y_E=\emptyset$ or $w_m\in Y_E$, and at most $d-k$ sets $X_i$ have odd cardinality. 
\end{thm}

Shephard's Theorem thus says that the absolute position of points on the moment curve is irrelevant and only their relative order is important. Furthermore, note that all points in $W$ are vertices of $P$. The following corollary helps in proving that two given edges of a hypergraph intersect properly.

\begin{cor}
\label{cor:shephardGen}
In the setting of Theorem~\ref{thm:shephard} assume that $W=U_1\cup U_2$ where $U_1$ and $U_2$ are embedded edges of a $k$-uniform hypergraph. Then these edges do not intersect in a way forbidden by Definition~\ref{def:dembed}, if there exists $j\in\{1,2\}$ such that
\[U_j=Y_S\cup X_1\cup\cdots\cup X_t\cup Y_E\]
holds where at most $d-k$ of the contiguous sets $X_i$ have odd cardinality.
\end{cor}

\begin{proof}
The two edges $U_1$ and $U_2$ do not intersect in a way forbidden by Definition~\ref{def:dembed} if at least one of them is a face of $P=\conv W$, which is the case for $U_j$. 
\end{proof}

\begin{prop}
\label{pro:shephardPlane}
Let $A$, $B$, $C$, and $D$ be four distinct points on the moment curve in $\RR^3$ in arbitrary order. Then the line segments $AB$ and $CD$ do not intersect.
\end{prop}

\begin{proof}
This follows immediately from Corollary~\ref{cor:shephardGen} for the case $k=2$ and $d=3$.
\end{proof}

In the $k=d=3$ case Corollary~\ref{cor:shephardGen} allows zero odd sets $X_i$. Thus, we can easily classify all possible configurations for two edges.

\definecolor{ablue}{cmyk}{1,0,0,.5}
\definecolor{ayell}{cmyk}{0,0,1,.6}
\definecolor{ared}{cmyk}{0,1,.5,.2}

\begin{table}[htbp]
	\centering
		\begin{tabular}{llll}
		\arrayrulecolor{black}
		\toprule
		\textbf{Nr.}  & \textbf{Configuration} & \textbf{Nr.} & \textbf{Configuration}\\
		\midrule
		\textcolor{ablue}{1} & E E E F F F&\textcolor{ablue}{9}& E I F F E\\
		\arrayrulecolor{agray}
		\hline
		\textcolor{ablue}{2} & E E F F E F&\textcolor{ablue}{10}& E F I E F\\
		\hline
		\textcolor{ablue}{3} & E E F F F E&\textcolor{ablue}{11}&E F E I F\\
		\hline
		\textcolor{ablue}{4} & E F F E E F&\textcolor{ayell}{12}& E E F E F F\\
		\hline
		\textcolor{ablue}{5} & E F F E I&\textcolor{ared}{13}& E F E F E F\\
		\hline
		\textcolor{ablue}{6} & E E I F F&\textcolor{ared}{14}& E F E F F E\\
		\hline
		\textcolor{ablue}{7} & E I E F F&\textcolor{ared}{15}& E F E F I\\
		\hline
		\textcolor{ablue}{8}& E E F F I&\textcolor{ared}{16}& E F I F E\\
		\arrayrulecolor{black}
		\bottomrule		
		\end{tabular}
	\caption{Possible configurations for two edges $e$ and $f$ on the moment curve in $\RR^3$ sharing at most one vertex. The vertices of $e\backslash f$ are marked with E, those of $f\backslash e$ marked with F, and a joint vertex is marked with I. Equivalent cases, one being the reverse of the other, are only displayed once.}
	\label{tab:config3}
\end{table}

\begin{lem}
\label{lem:shephard4}
Let $H$ be a 3-uniform hypergraph and $\phi:V(H)\to\RR^3$ such that $\phi$ maps all vertices one-to-one on the moment curve and for each pair of edges $e$ and $f$ sharing at most one vertex, the order of the points $\phi(e\cup f)$ on the moment curve has one of the Configurations~1--12 shown in Table~\ref{tab:config3}. Then $\phi$ is an embedding of $H$. 
\end{lem}

\begin{proof}
Note that the relative order of edges with two common vertices is irrelevant as they always intersect according to Definition~\ref{def:dembed}. Configurations~1--11 follow directly from Corollary~\ref{cor:shephardGen} for $k=d=3$. Thus, we are left with Configuration~12 and it is sufficient to prove the following: For $x_{0,0}<x_{1,0}<x_{0,1}<x_{2,0}<x_{1,1}<x_{2,1}\in\RR$, $\psi:\RR\to\RR^3,\psi(x)=(x,x^2,x^3)$ the moment curve, and $D_i=\{x_{0,i},x_{1,i},x_{2,i}\}$ we have that $\conv\psi(D_0)\cap\conv\psi(D_1)=\varnothing$. Assume otherwise. Note that if two triangles intersect in $\RR^3$ the intersection points must contain at least one point of the border of at least one of the triangles. Thus, without loss of generality, $\conv\{\psi(x_{j_1,0}),\psi(x_{j_2,0})\}\cap\conv\psi(D_1)\neq\varnothing$. However, by Theorem~\ref{thm:shephard} we know that $\conv\{\psi(x_{j_1,0}),\psi(x_{j_2,0})\}$ is a face of the polytope $P=\conv(\{\psi(x_{j_1,0}),\psi(x_{j_2,0})\}\cup\psi(D_1))$ which is a contradiction.
\end{proof}

Note that if we have two edges with vertices on the moment curve as in Configurations~13--16 they generally \emph{do} intersect in a way forbidden by Definition~\ref{def:dembed}. Also, we have presented above all possible cases for the relative order of vertices of two edges on the moment curve. Not all of them will actually be needed in the proofs of the next section.


\section{Bounding the weak chromatic number}
\label{sec:weak}

For $d,k,n\in\NN$ we define
\[\chi_{d,k}\wea(n)=\max\{\chi\wea(H):H\in\embH{d}{k},\mach{V(H)}=n\}\]
to be the maximum weak chromatic number of a $d$-embeddable $k$-uniform hypergraph on $n$ vertices.

In this section, we give lower and upper bounds on $\chi_{d,k}\wea(n)$.
Obviously, 
$\chi_{d,k}\wea(n)$ is monotonically increasing in $n$ and in $d$ and monotonically decreasing in $k$ if the other parameters remain fixed.

\begin{remark}
a) For $k=2$, the results in Tables~\ref{tab:currWeakL} and \ref{tab:currWeakU} follow from the Four Color Theorem and the fact that all graphs are $d$-embeddable for $d\geq 3$.

b) For $d\geq2k-1$, we have $\chi_{d,k}\wea(n)=\lceil n/(k-1)\rceil$ as $K_n^{(k)}$ is $(2k-1)$-embeddable for all $k\in\NN$ by the Menger-Nöbeling Theorem (see \cite[p. 295]{Menger} and \cite{Nobeli}) and $\chi\wea\left(K_n^{(k)}\right)=\lceil n/(k-1)\rceil$.

c) For $d\leq k-2$, we again know $\chi_{d,k}\wea(n)=1$ as $H\in\embH{d}{k}$ cannot have any edge.
\end{remark}

\begin{prop}
For all $n\geq3$ we have $\chi_{2,3}\wea(n)\leq2$. (This bound is obviously sharp.)
\end{prop}

\begin{proof}
Let $H\in\embH{2}{3}$ and $V=V(H)$. Then $G=\shadow(H)$ is a planar graph, thus $\chi(G)\leq 4$. Let $\kappa:V\to\{1,2,3,4\}$ be a 4-coloring of $G$. Define 
\[\kappa':V\to\{1,2\},v\mapsto(\kappa(v)\mod 2)+1.\]
In any triangle $\{u,v,w\}$ of $H$ under the coloring $\kappa$ these vertices have exactly three different colors. Therefore, under the coloring $\kappa'$ at least one vertex with color 1 and one vertex with color 2 exists. Thus $\kappa'$ is a valid 2-coloring of $H$.
\end{proof}

\begin{thm}
\label{thm:wup1}
Let $d\geq3$. Then one has
\[\chi_{d,d}\wea(n)\leq\left\lceil\left(\frac{6ed}{(d-1)!}\right)^{\frac{1}{d-1}}n^{\frac{d-2}{d-1}}\right\rceil=\Ol\left(n^{\frac{d-2}{d-1}}\right)\quad\text{as }n\to\infty.\]
This result also holds for piecewise linear embeddings.
\end{thm}

\begin{proof}
Let $H\in\embHP{d}{d}\supseteq\embH{d}{d}$. By Corollary~\ref{cor:sizeH} we know that every edge is adjacent to at most $\Delta=d(6n^{d-2}-12n^{d-3})/(d-1)!-d$ other edges.

We want to apply the Lovász Local Lemma \cite{local,local2} to bound the weak chromatic number of $H$. Let $c\in\NN$. In any $c$-coloring of the vertices of $H$ an edge is called bad if it is monochromatic and good if not. In a uniformly random $c$-coloring the probability for any one edge to be bad is $p=\frac{1}{c^{d-1}}$. Moreover, let $e$ be any edge in $H$ and $F$ be the set of edges in $H$ not adjacent to $e$. Then the events of $e$ being bad and of any edges from $F$ being bad are independent. Thus the event whether any edge is bad is independent from all but at most $\Delta$ other such events.

The Lovász Local Lemma guarantees us that with positive probability all edges are good if $e\cdot p\cdot(\Delta+1)\leq1$. This implies that $H$ is weakly $c$-colorable. Note that
\[e\cdot p\cdot(\Delta+1)\leq1\Leftrightarrow \frac{ed(6n^{d-2}-12n^{d-3})}{(d-1)!}-ed+e\leq c^{d-1}.\]
Choosing an integer
\[c\geq\left(\frac{6ed}{(d-1)!}\right)^{\frac{1}{d-1}}n^{\frac{d-2}{d-1}}\geq\left(\frac{ed(6n^{d-2}-12n^{d-3})}{(d-1)!}-ed+e\right)^{\frac{1}{d-1}},\]
the hypergraph $H$ is $c$-colorable and $\chi\wea(H)\leq c$.
\end{proof}

\begin{thm}
\label{thm:wup2}
Let $d\geq \ell\geq 3$. Then one has
\[\chi_{2d-\ell,d}\wea(n)\leq\left\lceil\left(\frac{e(d-\ell+2)!}{(d-1)!}\right)^{\frac{1}{d-1}}n^{1-\frac{3^{\ell-1-d}}{d-1}}\right\rceil=\Ol\left(n^{1-\frac{3^{\ell-1-d}}{d-1}}\right)\quad\text{as }n\to\infty.\]
This result also holds for piecewise linear embeddings.
\end{thm}

\begin{proof}
By Corollary~\ref{cor:sizeGundert} we know that every edge is adjacent to at most $\frac{(d-\ell+2)!}{(d-1)!}\cdot n^{d-1-3^{\ell-1-d}}-d$ other edges. The rest of the proof is now analogous to the proof of Theorem~\ref{thm:wup1}.
\end{proof}

\begin{thm}
\label{thm:wup3}
Let $d\geq2$. Then one has
\[\chi_{d-1,d}\wea(n)\leq\left\lceil(ed)^{\frac{1}{d-1}}n^{\frac{\lceil(d-1)/2\rceil}{d-1}}\right\rceil=\left\{\begin{array}{ll}
\Ol\left(n^{1/2}\right)&\text{if $d$ is odd}\\
\Ol\left(n^{1/2+1/(2d-2)}\right)&\text{if $d$ is even}
\end{array}\right.\quad\text{as }n\to\infty.\]
\end{thm}

\begin{proof}
By Corollary~\ref{cor:deypach} we know that every edge is adjacent to at most $dn^{\lceil (d-1)/2\rceil}-1$ other edges. The rest of the proof is now analogous to the proof of Theorem~\ref{thm:wup1}.
\end{proof}

By monotonicity, the upper bounds presented here also hold if the uniformity of the hypergraph is larger than stated in Theorems~\ref{thm:wup1}~and~\ref{thm:wup2}. In the remaining part of this section, we now consider lower bounds for the weak chromatic number of hypergraphs.

\begin{thm}
\label{thm:wlo1}
For $n\geq 2$ we have 
 \[\chi_{3,3}\wea(n)\geq\frac{\log n}{2\log\log n}-1=\Omega\left(\frac{\log n}{\log\log n}\right)\quad\text{as }n\to\infty.\]
\end{thm}

\begin{proof}
We first define a sequence of hypergraphs $H_m$ for $m\geq 2$ such that $\chi\wea(H_m)\geq m$.
Set $H_2=K_3^{(3)}$ which has 3 vertices.
Define $H_m$ for $m>2$ iteratively, assuming $\chi\wea(H_{m-1})\geq m-1$.
Take $m$ new vertices $\{v_0,\ldots,v_{m-1}\}$ and $m(m-1)/2$ disjoint copies of $H_{m-1}$, labeled $H_{m-1}^{[0,1]},\ldots,H_{m-1}^{[m-2,m-1]}$.

The edges of $H_m$ shall be all former edges of all $H_{m-1}^{[i,j]}$ together with all edges of the form $\{v_i,v_j,w\}$ where $i<j$ and $w\in H_{m-1}^{[i,j]}$. Assume $H_m$ is weakly $(m-1)$-colorable. Given such a coloring, one color must occur twice in $\{v_0,\ldots,v_{m-1}\}$. Say, these are the vertices $v_{i_1}$ and $v_{i_2}$ where $i_1<i_2$. This color cannot occur anymore in the coloring of $H_{m-1}^{[i_1,i_2]}$. Thus, $H_{m-1}^{[i_1,i_2]}$ must be weakly $(m-2)$-colorable. This is a contradiction and $H_m$ is at least (and obviously exactly) weakly $m$-chromatic.

\begin{figure}[htbp]
	\centering
		\includegraphics[scale=\scalefactor]{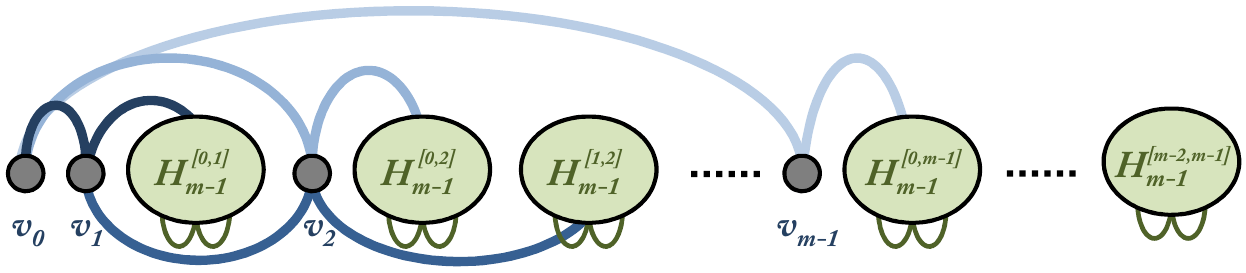}
	\caption{Construction of $H_m$.}
	\label{fig:Hm}
\end{figure}

We now claim that $H_m\in\embH{3}{3}$ for all $m\geq 2$. For that we give a function $f_m:V(H_m)\to\{1,\ldots,n_m\}$ where $n_m$ is the number of vertices of $H_m$. This function defines the order in which the vertices of $H_m$ will be arranged on the moment curve $t\mapsto(t,t^2,t^3)$. Lemma~\ref{lem:shephard4} on possible configurations then guarantees that $H_{m}$ is embeddable via arbitrary points on the moment curve. Note that the absolute position of vertices on the moment curve is not important, only their relative order.

The hypergraph $H_2=K_3^{(3)}$ can be embedded into $\RR^3$ via any three points on the moment curve, so $f_2:V(H_2)\to\{1,2,3\}$ can be chosen arbitrarily. Assume that $f_{m-1}$ has already been defined and that the vertices of $H_{m-1}$ arranged in that order on the moment curve form an embedding. Look at the vertices of $H_m$ as given before. We define $f_m(v_j)=n_{m-1}\cdot j(j-1)/2+j$ for $0\leq j\leq m-1$ and for any $w\in H_{m-1}^{[i,j]}$ with $i<j$ we set $f_m(w)=n_{m-1}\cdot (j(j-1)/2+i)+j+f_{m-1}(w)$. This gives exactly the order shown in Figure~\ref{fig:Hm}.

Now, arrange the vertices of $H_{m}$ on the moment curve in that order and pick any two edges $e_1$ and $e_2$. By Lemma~\ref{lem:shephard4} we can assume that they do not share two vertices.

\textit{Case~1: $e_1$ and $e_2$ are from the same subhypergraph $H_{m-1}^{[i,j]}$.} Then, by induction, they can only intersect according to Definition~\ref{def:dembed} as their relative order reflects that of $f_{m-1}$.

\textit{Case~2: $e_1$ and $e_2$ are from distinct subhypergraphs $H_{m-1}^{[i_1,j_1]}$ and $H_{m-1}^{[i_2,j_2]}$.} Then we are in Case~1 in Table~\ref{tab:config3} and thus they intersect according to Definition~\ref{def:dembed}.

\textit{Case~3: $e_1=\{v_{i_1},v_{j_1},v\}$ where $v\in H_{m-1}^{[i_1,j_1]}$ and $e_2$ is from some subhypergraph $H_{m-1}^{[i_2,j_2]}$.} Without loss of generality, let $e_2=\{w_1,w_2,w_3\}$ and assume  that $f_m(w_1)<f_m(w_2)<f_m(w_3)$. Then, by definition, $i_1<j_1$ and $i_2<j_2$ and all the possible cases of Lemma~\ref{lem:shephard4} are listed in Table~\ref{tab:case1}.

\begin{table}[htbp]
	\centering
		\begin{tabular}{lll}
		\arrayrulecolor{black}
		\toprule
		Relative order of $v$ & Additional Condition& Case number \\
		\midrule
		$f_m(v)<f_m(w_1)$&-&1\\
		$f_m(v)=f_m(w_1)$&-&6\\
		$f_m(w_1)<f_m(v)<f_m(w_2)$&-&12\\
		$f_m(v)=f_m(w_2)$&-&7\\
		$f_m(w_2)<f_m(v)<f_m(w_3)$&-&2\\
		$f_m(v)=f_m(w_3)$&-&8\\
		$f_m(v)>f_m(w_3)$&$f_m(v_{i_1})<f_m(w_3)$&3\\
		$f_m(v)>f_m(w_3)$&$f_m(v_{i_1})>f_m(w_3)$&1\\
		\bottomrule
		\end{tabular}
	\caption{Sub-cases of Case~3 in the proof of Theorem~\ref{thm:wlo1} referring to the corresponding cases of Lemma~\ref{lem:shephard4}.}
	\label{tab:case1}
\end{table}

\textit{Case~4: $e_1=\{v_{i_1},v_{j_1},v\}$ and $e_2=\{v_{i_2},v_{j_2},w\}$.} Again, $i_1<j_1$ and $i_2<j_2$ holds. Without loss of generality assume $j_1\leq j_2$.
We then have one of the cases listed in Table~\ref{tab:case2}.
\begin{table}[htbp]
	\centering
		\begin{tabular}{ll}
		\arrayrulecolor{black}
		\toprule
		Relative order of $i_1,i_2,j_1,j_2$ & Case number \\
		\midrule
		$j_1=j_2$ and $i_1\neq i_2$& 10\\
		$j_1=j_2$ and $i_1=i_2$ & two shared vertices\\
		\arrayrulecolor{agray}
		\hline
		$i_1<j_1<i_2<j_2$& 1\\
		$i_1<j_1=i_2<j_2$& 7\\
		$i_1<i_2<j_1<j_2$& 2\\
		$i_1=i_2<j_1<j_2$& 8\\
		$i_2<i_1<j_1<j_2$ & 3\\
		\arrayrulecolor{black}
		\bottomrule
		\end{tabular}
	\caption{Sub-cases of Case~4 in the proof of Theorem~\ref{thm:wlo1} referring to the corresponding cases of Lemma~\ref{lem:shephard4}.}
	\label{tab:case2}
\end{table}

Thus, the order given by $f_m$ provides an embedding of $H_m$. To estimate $n_m$, we use the following recursion
	\begin{align*}
		n_2&=3,\\
		n_m&=m+n_{m-1}\cdot m(m-1)/2\quad \text{for }m>2.
  \end{align*} 
This can be bounded by $n_m\leq m^{2m}=:\hat{n}_m$.
Then \[\frac{\log \hat{n}_m}{\log\log \hat{n}_m} = 2m\cdot\frac{\log m}{\log(2m\log m)}\leq2m\]
and we finally get that
$m\geq \frac{\log \hat{n}_m}{2\log\log \hat{n}_m}\geq\frac{\log n_m}{2\log\log n_m}$.
\end{proof}

Note that by monotonicity also $\chi_{4,3}\wea(n)=\Omega\left(\frac{\log n}{\log\log n}\right)$ holds.


\begin{thm}
\label{thm:wlo2}
Let $d\geq3$. For $n\geq d$ we have
\[\chi_{2d-3,d}\wea(n)\geq\frac{\log n}{2\log\log n}-\frac{d-1}2=\Omega\left(\frac{\log n}{\log\log n}\right)\quad\text{as }n\to\infty.\]
\end{thm}

\begin{proof}
Induction over $d$. The case $d=3$ was shown in Theorem~\ref{thm:wlo1}. Let $d>3$. Suppose we have constructed a family $(H_m^{d-1})_{m\in\NN}$ of hypergraphs in $\embH{2d-5}{d-1}$  such that $\chi\wea(H_m^{d-1})\geq m$ and such that all hypergraphs $H_m^{d-1}$ are embeddable into $\RR^{d-1}$ by vertices on the moment curve with edges intersecting according to Corollary~\ref{cor:shephardGen} (or Lemma~\ref{lem:shephard4} if $d=4$).
 	
Let $H_2^d=K_d^{(d)}$. The hypergraph $H_2^d$ has $d$ vertices, one edge, and is weakly 2-colorable.
Define $H_m^d$ for $m>2$ iteratively, given that $\chi\wea(H_{m-1}^d)\geq m-1$.
For that take one copy of $H_{m-1}^d$ and one copy of $(d-1)$-uniform $H_m^{d-1}$.

The edges of $H_m^d$ shall be all edges of $H_{m-1}^d$ and all edges of the form $(\{v\}\cup e)$ for $v\in V(H_{m-1}^d)$ and $e\in E(H_m^{d-1})$.
Assume that there exists a weak $(m-1)$-coloring of $H_m^d$. Then there has to be at least one monochromatic edge $e\in E(H_m^{d-1})$.
No vertex of $H_{m-1}^d$ can be colored with this color, so its edges must be weakly $(m-2)$-colored. This is a contradiction and thus $\chi\wea(H_m^d)\geq m$.

\begin{figure}[htbp]
	\centering
		\includegraphics[scale=\scalefactor]{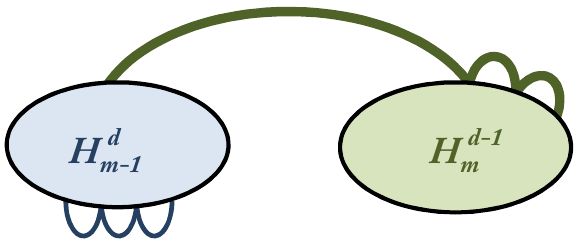}
	\caption{Construction of $H_m^d$.}
	\label{fig:Hmd}
\end{figure}

We now claim that $H_m^d\in\embH{2d-3}{d}$ for all $m\geq 2$. As in the proof of Theorem~\ref{thm:wlo1}, we give a function $f_m^{(d)}:V(H_m^{d})\to\{1,\ldots,n_{m}^{(d)}\}$ where $n_m^{(d)}$ is the number of vertices of $H_m^d$. This defines the order in which the vertices of $H_m^d$ will be arranged on the moment curve $t\mapsto(t,\ldots,t^{2d-3})$. We then use Corollary~\ref{cor:shephardGen} to prove that $H_{m}^d$ is embeddable via arbitrary points on the moment curve. As before, the absolute position of vertices on the moment curve is not important. For a fixed uniformity $d$ and dimension $2d-3$, Corollary~\ref{cor:shephardGen} guarantees that if for two given edges the vertices of at least one edge have at most $d-3$ odd contiguous subsets, they intersect properly according to Definition~\ref{def:dembed}.

For $d=3$ we can set $f_m^{(3)}=f_m$ for all $m\geq 2$, where $f_m$ is as in the proof of Theorem \ref{thm:wlo1}. For $d>3$ we have by assumption that there exists a corresponding family of functions
\[\left(f_m^{(d-1)}:V(H_m^{d-1})\to\{1,\ldots,n_{m}^{(d-1)}\}\right)_m\]
such that the vertices of $H_{m}^{d-1}$ arranged in that order on the moment curve form an embedding. We then have to give an appropriate family of functions $f_m^{(d)}$ for $d$.

$H_2^d$ can be embedded into $\RR^{2d-3}$ via any $d$ points on the moment curve, so $f_2^{(d)}:V(H_2^d)\to\{1,\ldots,d\}$ can be chosen arbitrarily.
Assume that $f_{m-1}^{(d)}$ has already been defined and gives an embedding of $H_{m-1}^{d}$. We define $f_m^{(d)}(v)=f_{m-1}^{(d)}(v)$ for $v\in V(H_{m-1}^d)$ and for any $w\in V(H_m^{d-1})$ we set $f_m^{(d)}(w)=n_{m-1}^{(d)}+f_m^{(d-1)}(w)$. This is also shown in Figure~\ref{fig:Hmd}.

Arrange the vertices of $H_{m}^d$ on the moment curve in that order and pick any two edges $g_1$ and $g_2$.

\textit{Case~1: Both edges are from the subhypergraph $H_{m-1}^{d}$.} Then they intersect in accordance to Definition~\ref{def:dembed} and Corollary~\ref{cor:shephardGen} as their relative order reflects that of $f_{m-1}^{(d)}$.

\textit{Case~2: One edge is from $H_{m-1}^{d}$ and the other of the form $(\{v\}\cup e)$ where $v\in V(H_{m-1}^{d})$ and $e\in E(H_{m}^{d-1})$}. Then both edges have at most one odd contiguous subset (besides the first and last one), which is no problem for $d>3$.

\textit{Case~3: $g_1=(\{v_1\}\cup e_1)$ and $g_2=(\{v_2\}\cup e_2)$.} Then the edges $e_1$ and $e_2$ intersect according to Corollary~\ref{cor:shephardGen} (or Lemma~\ref{lem:shephard4} if $d=4$) and $g_1$ and $g_2$ have at most one more odd contiguous subset than the edges $e_1$ and $e_2$ had in the ordering of $f_m^{(d-1)}$. The last number, by assumption, was bounded from above by $(d-1)-3$ for at least one $e_i$, $i\in\{1,2\}$ (unless $d=4$ and they intersect according to Case~12 in Table~\ref{tab:config3}, see below). So at least one $g_i$ has at most $d-3$ odd contiguous subsets. Thus, the order given by $f_m^{(d)}$ provides an embedding of $H_m^d$.

Note that there is one small exception to Case~3 when $d=4$. Here, $e_1$ and  $e_2$ could be in the relative position of Case~12 in Table~\ref{tab:config3} and consequently have more than $(d-1)-3=0$ odd contiguous subsets. However, this is no problem as in all possible extensions to $g_1$ and $g_2$ at least one of the edges continues to have only one odd contiguous subset (see Table~\ref{tab:configD}).

\begin{table}[htbp]
	\centering
		\begin{tabular}{ll}
		\arrayrulecolor{black}
		\toprule
		\textbf{Nr.}  & \textbf{Configuration}\\
		\midrule
		\textcolor{ablue}{1} &E F E E F E F F\\
		\arrayrulecolor{agray}
		\hline
		\textcolor{ablue}{2} &F E E E F E F F\\
		\hline
		\textcolor{ablue}{3} &I E E F E F F\\
		\arrayrulecolor{black}
		\bottomrule		
		\end{tabular}
	\caption{All possible 4-uniform extensions of Case~12 in Table~\ref{tab:config3} as occurring in the construction of $H_m^4$.}
	\label{tab:configD}
\end{table}

To bound the number of vertices of $H_m^d$ we use
	\begin{align*}
		n_2^{(d)}&=d,\\
		n_m^{(d)}&=n_{m-1}^{(d)}+n_{m}^{(d-1)}\quad \text{for }m>2.
  \end{align*}
Iteratively, we get that $n_m^{(d)}=d+\sum_{r=3}^m n_r^{(d-1)}\leq m\cdot n_m^{(d-1)}\leq\cdots\leq m^{d-3}\cdot \hat{n}_m=m^{2m+d-3}$ and thus
\[\frac{\log n_m^{(d)}}{\log\log n_m^{(d)}} \leq (2m+d-3)\cdot\frac{\log(m)}{\log\left((2m+d-3)\log(m)\right)}\leq 2m+d-3.\]
Hence,
$m\geq\frac{\log n_m^{(d)}}{2\log\log n_m^{(d)}}-\frac{d-3}2$.
\end{proof}

Note that by monotonicity also $\chi_{2d-2,d}\wea(n)=\Omega\left(\frac{\log n}{\log\log n}\right)$ holds.

\section{Conclusions and open questions}
Starting from the Four Color Theorem we have shown that it has no direct analogon for higher dimensions in general. Rather, in almost all cases, the number of colors needed to color a hypergraph embedabble in a certain dimension is unbounded. However, some questions still need to be answered.

Firstly, it would be very interesting to see whether the logarithmic-polynomial difference between lower and upper bounds for the weak coloring case can be improved substantially. If the conjectures by Gundert and Kalai mentioned in Section~\ref{sec:embed} were true, the upper bound for weak colorings could be lowered as follows.

\begin{conj}
Let $k-1\leq d\leq 2k-2$. Then one has
\[\chi_{d,k}\wea(n)=\Ol\left(n^{\frac{\lceil(d-1)/2\rceil}{k-1}}\right)\quad\text{as }n\to\infty.\]
\end{conj}

Further, in the weak coloring case, for $k=d+1$ no examples with an unbounded number of colors needed have yet been found and a finite bound is still possible. Also, the question whether the maximum chromatic number for some fixed $k$, $d$, and $n$ actually differs for linear and piecewise linear embeddings, remains an open problem.

\section*{Acknowledgments}
The authors wish to thank Penny Haxell for helpful discussions. They also would like to thank two anonymous
referees for careful and valuable remarks concerning the presentation of this work, in particular considerably simplifying the treatment of strong colorings.

\bibliographystyle{amsplain} 	 
\bibliography{hyper}	 		 	 	 
\clearpage

\end{document}